\documentclass[a4paper]{amsart}
\usepackage[utf8]{inputenc}
\usepackage[english]{babel}
\usepackage{amsfonts}
\usepackage{amssymb}
\usepackage{amsthm}
\usepackage{mathtools}
\usepackage{bm}
\usepackage[left=3cm, right=3cm, top=3cm, bottom=3cm]{geometry}

\usepackage{mathtools}
\mathtoolsset{showonlyrefs}

\usepackage{hyperref}
\usepackage{xcolor}
\hypersetup{
	colorlinks,
	linkcolor={blue!90!black},
	citecolor={green!40!black},
	urlcolor={blue!90!black}
}

\usepackage[backend=bibtex,backref,style=numeric,doi=false,url=false,isbn=false,sorting=nyt,sortcites=true,giveninits=true]{biblatex}
\renewbibmacro{in:}{}
\usepackage{csquotes}
\newcommand{\Z}{\mathbb Z}

\newcommand{\R}{\mathbb R}
\newcommand{\C}{\mathbb C}

\newcommand{\mc}{\mathcal}

\renewcommand{\phi}{\varphi}
\newcommand{\n}{\nabla}
\newcommand{\pd}{\partial}
\newcommand{\tr}{{\rm tr}}

\renewcommand{\leq}{\leqslant}
\renewcommand{\geq}{\geqslant}

\newcommand{\mf}{\mathfrak}
\newcommand{\mb}{\mathbf}

\newcommand{\Sym}{\mathrm{Sym}\,}

\newcommand{\la}{\langle}
\newcommand{\ra}{\rangle}
\newcommand{\til}[1]{\widetilde{#1}}

\newcommand{\Ric}{\mathrm{Ric}}
\newcommand{\Rm}{\mathrm{Rm}}

\renewcommand{\bar}[1]{\mkern 0.5mu\overline{\mkern-0.5mu#1\mkern-0.5mu}\mkern 0.5mu}

\theoremstyle{plain}
\newtheorem{theorem}{Theorem}[section]
\newtheorem*{theorem*}{Theorem}
\newtheorem{conjecture}[theorem]{Conjecture}
\newtheorem{lemma}[theorem]{Lemma}

\newtheorem{corollary}[theorem]{Corollary}
\newtheorem{proposition}[theorem]{Proposition}

\theoremstyle{definition}
\newtheorem{remark}[theorem]{Remark}
\newtheorem{ex}[theorem]{Example}
\newtheorem{definition}[theorem]{Definition}
\newtheorem{example}[theorem]{Example}

\newtheorem{innercustomgeneric}{\customgenericname}
\providecommand{\customgenericname}{}
\newcommand{\newcustomtheorem}[2]{%
	\newenvironment{#1}[1]
	{%
		\renewcommand\customgenericname{#2}%
		\renewcommand\theinnercustomgeneric{##1}%
		\innercustomgeneric
	}
	{\endinnercustomgeneric}
}
\newcustomtheorem{customthm}{Theorem}
\newcustomtheorem{customprop}{Proposition}
\newcustomtheorem{customcor}{Corollary}

\newcommand{\T}{\bm T}
\newcommand{\D}{\mathbb D}

\renewcommand{\i}{\mathrm{i}}

\begin{document}
	
	\title{On geometry of steady toric K\"ahler-Ricci solitons}
	\author{Yury Ustinovskiy}
	\email{yuraust@gmail.com}
	\address{Lehigh University, Department of Mathematics\\
	Chandler-Ullmann Hall\\
	17 Memorial Drive East\\
	Bethlehem, PA 18015}
	\date{}
	\begin{abstract}
		In this paper we study the gradient steady K\"ahler-Ricci soliton metrics on non-compact toric manifolds. We show that the orbit space of the free locus of such a manifold carries a natural Hessian structure with a nonnegative Bakry-\'Emery tensor. We generalize Calabi's classical rigidity result and use this to prove that any complete
		$\T^n$-invariant gradient steady K\"ahler-Ricci soliton with a free torus action must be a flat $(\C^*)^n$.
	\end{abstract}
	\maketitle

	\section{Introduction}\label{s:intro}

	A Riemannian manifold $(M,g)$ is a \textit{Ricci soliton}, if its Ricci curvature satisfies the equation
	\[
	\Ric+\frac{1}{2}\mc L_V g=\lambda g
	\]
	for some vector field $V\in C^\infty(M,TM)$ and a constant $\lambda\in \R$. Ricci solitons are the stationary solutions to the Ricci flow modulo the action of the diffeomorphism group and scaling. They represent singularity models for the flow and realize a class of distinguished metrics generalizing the notion of the Einstein metrics.
	A Ricci soliton $(M,g)$ is \textit{expanding, steady or shrinking} depending on whether $\lambda<0$, $\lambda=0$ or $\lambda>0$. A particularly important class of \textit{gradient} Ricci solitons occurs when $V=\n_g f$ is a gradient vector field. In this case, the soliton equation takes form
	\[
	\Ric+\n^2_gf=\lambda g.
	\]
	Since the early works of Hamilton on the Ricci flow~\cite{ha-82,ha-86,ha-93}, classification and construction of the gradient Ricci solitons has been one of the central problems in geometric analysis, see, e.g.,~\cite{iv-93, pe-wy-09, cao-10,mu-se-13}. In the present paper we will focus on \textit{steady gradient} Ricci solitons. On such solitons $\Delta f- |df|^2=const$, and it follows from the maximum principle that a \textit{compact} steady Ricci soliton is necessarily Einstein. Thus if we are interested in genuine solitons (i.e., with non-Killing soliton vector field $V=\n_g f$) we have to consider non-compact backgrounds. The next best thing is $(M,g)$ being \textit{complete}~--- this set up is particularly important for the study of singularity models for the Ricci flow.

	Recall that a Riemannian manifold $(M,g,J)$ equipped with an integrable complex structure $J\colon TM\to TM$, $J^2=-\mathrm{Id}$ is \textit{K\"ahler} if $\n_gJ=0$, or, equivalently, if the skew-symmetric form $\omega=gJ$ is closed. The Ricci flow is known to preserve the K\"ahler condition, so it is particularly important to study and classify Ricci solitons on K\"ahler backgrounds~--- such structures are referred to as \textit{K\"ahler-Ricci solitons} or KRS for short. In the last decades, the question of the existence of K\"ahler-Einstein and KRS metrics has been related to the fundamental notions of stability in algebraic geometry.

	In this paper we study the classification of the \textit{steady} KRS on complete backgrounds. There exist numerous constructions providing explicit examples of complete steady KRS manifolds (including K\"ahler-Ricci flat). In~\cite{ti-ya-90} Tian and Yau constructed a complete K\"ahler-Ricci flat metric on the complement of a canonical divisor in a Fano manifold. Anderson, Kronheimer and LeBrun~\cite{an-kr-lb-89} constructed a family of complete K\"ahler-Ricci flat manifolds of infinite topological type. More recently Yang~\cite{ya-19} found an ``exotic'' complete K\"ahler-Ricci flat metric on $\C^3$ with a singular tangent cone at infinity. Existence of these examples suggests that a possible classification of complete K\"ahler-Ricci flat metrics even on $\C^n$ is a very difficult problem. The situation with the steady KRS is at least as subtle.
	
	In~\cite{ca-96} Cao constructed a unique complete non-flat $SO(2n)$-invariant steady KRS on $\C^n$ by reducing the soliton equation to an explicit ODE for a function of $|z|^2$. Recently Biquard and Macbeth~\cite{bi-mc-17} 
	used the gluing construction to desingularize Cao's KRS quotient metric on a crepant resolution of $\C^n/\Gamma$. Conlon and Deruelle~\cite{co-de-20} proved a similar statement using the continuity method.
	Sch\"afer~\cite{sc-21} adapted the ideas of~\cite{co-de-20} to construct steady solitons on crepant resolutions of orbifolds $(\C\times X_{\mathrm{CY}})/\Gamma$, where $X_{\mathrm{CY}}$ is a compact Calabi-Yau manifold.
	
	We stress that the known examples of the steady KRS are manifestly non-unique, e.g., the product of Cao's KRS $(\C^{n_1},g_1^{\mathrm{Cao}})\times (\C^{n_2},g_2^{\mathrm{Cao}})$ is never isomorphic to $(\C^{n_1+n_2},g^{\mathrm{Cao}})$. This is in contrast with the expanding and shrinking complete solitons, where the uniqueness is known under mild geometric assumptions, see, e.g., Cifarelli~\cite{ci-20} for the uniqueness of shrinking solitons with toric symmetry, and Conlon, Deruelle, Sun~\cite{co-de-su-19} for several uniqueness results of both shrinking and expanding solitons.
	
	All of the known \textit{uniqueness} results for steady solitons assume an explicitly prescribed asymptotics of the soliton vector field and the metric at infinity, see, e.g.,~\cite{co-de-20}. This situation is largely unsatisfactory, since, it does not answer even the simplest questions:  what are the possible soliton vector fields on $\C^n$? what are the possible asymptotics of a KRS on $\C^n$? does $(\C^*)^n$ admit any non-trivial steady KRS?
	
	To make an initial progress into the classification of steady KRS $(M,g,J)$ without a priori assumptions on the asymptotics of $g$ at infinity, we focus on \textit{toric} K\"ahler manifolds~--- manifolds $(M^{2n},g,J)$ equipped with an effective Hamiltonian action of a half-dimensional torus $\T^n=(S^1)^n$. Toric manifolds have been a particularly fruitful source of examples of K\"ahler and algebraic manifolds with explicitly computable underlying geometric invariants. Some of the early substantial advances in the problem of construction of canonical metrics has been made on \textit{toric} backgrounds, see,.e.g.,~\cite{wa-zh-04, zh-12, do-09}. The main advantage of this setting is that the geometry of a toric K\"ahler manifold $(M^{2n},g,J)$ can be fully captured by a simpler geometric data on a real $n$-dimensional orbit space $M^{2n}/\T^n$, which is essentially an object of convex geometry~\cite{de-88}. This allows to reduce the construction and classification of the canonical metrics on $(M,g,J)$ to the classical and better understood problems of real convex analysis.
	
	The main result of this paper is the following theorem classifying steady toric KRS with a free torus action. Essentially it is a non-existence result concluding that any such soliton must be a flat $(\C^*)^n$.
	
	\begin{customthm}{\ref{th:krs_uniqueness}}
		\it{
		Let $(M,g,J,\omega,\T^n)$ be a complete steady toric K\"ahler-Ricci soliton with the free torus action $\T^n\times M\to M$. Then $(M,g,J,\omega,\T^n)$ is equivariantly isomorphic to a flat algebraic torus $(\C^*)^n$ equipped with the standard complex structure and the action of $U(1)^n\subset (\C^*)^n$.
		}
	\end{customthm}

	Let us discuss the strategy of the proof. The orbit space $M^{2n}/\T^n$ of a free Hamiltonian action on a symplectic manifold $(M^{2n},\omega)$ has a natural \textit{affine structure}, i.e., there is a canonical torsion-free flat connection $D$ on $M^{2n}/\T^n$. A $\T^n$-invariant K\"ahler $(M^{2n},g,J)$ structure compatible with $\omega$ upgrades the affine structure on $N:=M^{2n}/\T^n$ to the \textit{Hessian structure}~-- a metric locally given by the Hessian of a convex function: $D^2u(x)$. With this extra structure, various notions of canonical metrics on $M^{2n}$ translate into scalar PDEs for the function $u(x)$ on $N$, see, e.g., Donaldson \cite{do-09} for the construction of constant scalar curvature metrics, Abreu~\cite{ab-98} for the explicit construction of extremal K\"ahler metrics, Zhu~\cite{zh-12} for the proof of the convergence of the K\"ahler-Ricci flow on toric Fano manifolds. Most importantly for us is that the \textit{steady KRS} equation turns out to be equivalent to the weighted Monge-Amp\`ere equation
	\begin{equation}\label{eq:intro_ma}
	\log\det \frac{\pd^2 u}{\pd x^i\pd x^j}=-v_px^p+\frac{\pd u}{\pd x^q}\xi^q+c
	\end{equation}
	where we denote by $\{x^i\}$ the local $D$-parallel coordinate system on $N$, and $v\in T^*N$, $\xi\in TN$ are fixed $D$-parallel sections. General Monge-Amp\`ere equations $\log\det u_{,ij}=F(x,u,\n u)$ have been central to the geometric analysis and the optimal transport. We refer the reader to the book of Guti{\'e}rrez~\cite{gu-16} and references therein for a detailed exposition. A special \emph{unweighted} case of equation~\eqref{eq:intro_ma} with $v=\xi=0$ was studied in the seminal paper of Calabi~\cite{ca-58}. In this paper Calabi proved the fundamental \textit{third order estimate} by providing an a priori upper bound for $|u_{,ijk}|^2$. The key input of Calabi is that if a convex function $u(x)$ solves the equation $\log\det u_{,ij}=c$ on $N$ then $(N,D^2u)$ has a nonnegative Ricci curvature. For the equation~\eqref{eq:intro_ma} this is not longer true, and Calabi's argument does not go through directly. However, if $u(x)$ solves~\eqref{eq:intro_ma}, then $(N,D^2u)$ has a nonnegative \textit{Bakry-\'Emery} tensor $\Ric_{\phi}:=\Ric+\n^2\phi$ for an appropriate weight function $\phi$.
	
	\begin{customprop}{\ref{prop:ric_positive}}(Kolesnikov,~\cite{ko-14})
		\it{
		Let $(N,D)$ be a simply-connected affine manifold equipped with a development map $\bm{\iota}\colon N\to\R^n$ and a Hessian metric $g=D^2u(x)$, such that function $u(x)$ solves the Monge-Amp\`ere equation
		\[
		\log\det \frac{\pd^2 u}{\pd x^i\pd x^j}=-v_px^p+\frac{\pd u}{\pd x^q}\xi^q+c\quad \xi\in\R^n, v\in (\R^n)^*.
		\]
		Consider the function
		\[
		\phi=\frac{1}{2}\left(u_{,p}\xi^p+v_qx^q\right).
		\]
		Then
		\[
		(\Ric_\phi)_{ij}=\frac{1}{4}g^{pq}g^{kl}u_{,ipk}u_{,jql}.
		\]
		In particular, $\Ric_\phi\geq 0$.
		}
	\end{customprop}
	This proposition is a special and simpler case of a lower bound on the Bakry-\'Emery Ricci tensor in~\cite{ko-14}, where Kolesnikov considered a more general Monge-Amp\`ere equation. With this crucial observation we manage to reproduce Calabi's estimates and get an upper bound on $|u_{,ijk}|^2_g$. This bound in turn yields the rigidity of \textit{complete} solutions to~\eqref{eq:intro_ma} implying that $u(x)$ must be a quadratic polynomial. In particular we have the following generalization of Calabi's classical result
	
	\begin{customcor}{\ref{cor:calabi_gen}}
		\it{
		Let $N\subset \R^n$ be an open subset and $u\colon N\to  \R$ a strictly convex function solving the Monge-Amp\`ere equation
		\begin{equation}
			\log\det (u_{,ij})=-v_jx^j+u_{,i}\xi^i+c,
		\end{equation}
		where $v\in (\R^n)^*$, $\xi\in\R^n$ and $c$ is a constant. If metric $g=D^2u$ on $N$ is complete, then $N=\R^n$, and $u(x)$ is a quadratic polynomial.
		}
	\end{customcor}

	The rest of the paper is organized as follows. In Section~\ref{s:bg} we review the necessary geometric background~--- Hessian manifolds, foundations of the K\"ahler toric geometry and the reduction of the steady KRS to the Monge-Amp\`ere equation on the quotient Hessian manifold. Most of the material of this section is standard and the results on toric K\"ahler structures can be found, e.g., in a foundational paper by Guillemin~\cite{gu-94}. In Section~\ref{s:bakry-emery} we make a seemingly unrelated detour into the geometry of smooth metric measure spaces with a non-negative Bakry-\'Emery Ricci tensor $\Ric_\phi$. Based on the results of Wei and Wylie~\cite{we-wy-09} we reprove the a priori estimates for the solution of the differential inequality $\Delta_\phi\sigma\geq \sigma^2$. Next, in Section~\ref{s:ma} we relate the general results on the properties of the non-negative Bakry-\'Emery Ricci tensor to the geometry of the Hessian manifolds solving the (weighted) Monge-Amp\`ere equations. In this section building on the work of Kolesnikov~\cite{ko-14} we establish the key fact that the Hessian manifolds underlying a steady toric KRS have a non-negative Bakry-\'Emery tensor for an appropriate weight function. With this observation we adapt the classical Calabi's third order estimate~\cite{ca-58}. Finally in Section~\ref{s:main} we apply the third order estimate to prove the main result, and formulate an open question regarding the properness of the moment map.
	
	\section*{Acknowledgments}\label{s:ack}
	
	This paper is a result of many fruitful discussions with Vestislav Apostolov and Jeffrey Streets. We are indebted to them for their interest and advice. We are grateful to thank Huai-Dong Cao for valuable suggestions at the earlier stages of this project. We also would like to thank Charles Cifarelli for pointing out a mistake in an earlier version of this paper.
	
	\section{Geometric Background}\label{s:bg}

	In this section we collect the necessary background on Hessian manifolds, K\"ahler toric manifolds and K\"ahler-Ricci solitons. We explain that the orbit space of a K\"ahler toric manifold inherits a Hessian structure, and give a scalar reduction of a steady toric K\"ahler-Ricci soliton to a weighted Monge-Amp\`ere equation on the orbit space.

	\subsection{Hessian Manifolds}\label{ss:hessian}

	\begin{definition}[Affine manifolds]
		\textit{An affine manifold} is a smooth manifold $N$
		equipped with a flat torsion free connection $D$ in the tangent bundle $TN$. The connection $D$ provides a distinguished coordinate system in a neighbourhood of any point, such that the transition functions between two such coordinate systems are affine.

		We can equivalently define an affine manifold as a smooth manifold $N$ equipped with an atlas $\phi_\alpha\colon U_\alpha\to \R^n$ such that the transition functions $\phi_\alpha\circ\phi_\beta^{-1}$ are affine. Atlas $(U_\alpha,\phi_\alpha)$ is called an \textit{affine structure} and uniquely determines a flat torsion free connection $D$.

		We denote by $\mc A$ the sheaf of \textit{affine functions} with the space of sections
		\[
		\Gamma(U,\mc A)=\{f\colon U\to \R\ |\ D^2f=0\}.
		\]
	\end{definition}

	\begin{example}\label{ex:affine_mfds}
		The simplest example of an affine manifold is $\R^n$ with the obvious flat connection. If a group $G\subset \mathrm{Aff}(\R^n)$ of affine transformations of $\R^n$ acts freely and properly-discontinuously on an open subset $U\subset \R^n$ then $U/G$ is an affine manifold. If $(M,D)$ is an affine manifold, and ${\bm\iota}\colon N\to M$ is an open immersion, then $M$ naturally inherits an affine structure $(N,\bm\iota^*D)$.
	\end{example}

	\begin{definition}[Development map]\label{rmk:affine_mfds}
		Let $(N,D)$ be an affine manifold, $\dim N=n$. An open map $\bm{\iota}\colon N\to \R^n$ is a \textit{development map} for $(N,D)$ if $D={\bm\iota}^*D_{0}$, where $D_{0}$ is the standard affine structure on $\R^n$. Equivalently, the 1-forms $\{{\bm\iota}^*dx^i\}$ provide a $D$-parallel trivialization of $T^*N$,
        where $\{x^1,\dots,x^n\}$ is a coordinate system on~$\R^n$.
	\end{definition}

    \begin{lemma}\label{lm:development_map}
        Any simply connected affine manifold $(N,D)$ admits a development map.
    \end{lemma}
    \begin{proof}
        Since $\pi_1(N)=1$ and the connection $D$ is flat, its holonomy is trivial: $\mathrm{Hol}(D)=\{\mathrm{id}\}$. Therefore one can construct a $D$-parallel trivialization $\{\eta^1,\dots,\eta^n\}$ of $T^*N$. Since $D$ is torsion free, and $D\eta^i=0$, we conclude that 1-forms $\eta^i$ are closed, and thus exact: $\eta^i=dx^i$ for $x^i\in C^\infty(N,\R)$. Then ${\bm\iota}=(x^1,\dots,x^n)$ is the required development map.
    \end{proof}

	\begin{definition}[Hessian manifolds]
		Let $(N,D)$ be an affine manifold. A function $u_\alpha$ on an open set $U_\alpha$ is called \textit{strictly convex} if the symmetric 2-tensor
		\[
		D^2u_\alpha\in C^\infty(U_\alpha, \mathrm{Sym^2}(T^*N))
		\]
		is strictly positive.

		A Riemannian metric $g$ on an affine manifold is called a \textit{Hessian metric} if there exists an open cover $N=\bigcup U_\alpha$ and a collection of strictly convex functions $u_\alpha\colon U_\alpha\to\R$ such that
		\begin{equation}\label{eq:hessian_def}
		g\big|_{U_\alpha}=D^2u_\alpha.
		\end{equation}
		An affine manifold $(N,D)$ with a Hessian metric $g=D^2u_\alpha$ is called a \textit{Hessian manifold}.
	\end{definition}

	A Hessian metric $g$ on an affine manifold $(N,D)$ defines an affine $\R$-bundle $L\to N$. Namely, if $g$ is determined by local convex functions $u_\alpha\colon U_\alpha\to\R$ so that~\eqref{eq:hessian_def} holds, then the differences $u_\alpha-u_\beta$ are affine functions on $U_{\alpha}\cap U_\beta$. Thus collection $\{u_\alpha-u_\beta\}$ defines a \v{C}ech 1-cocycle and the corresponding cohomology class
	\[
	[\{u_\alpha-u_\beta\}]\in H^1(N,\mc A),
	\]
	where the group $H^1(N,\mc A)$ classifies \textit{affine} $\R$-bundles over $N$. Given two sections of an affine $\R$-bundle $s_1,s_2\in\Gamma(N,L)$, their difference is a well-defined function
	\[
	s_1-s_2\in C^\infty(N,\R).
	\]

	It is easy to see that on a simply-connected affine manifold $(N,D)$ the group $H^1(N,\mc A)$ is trivial, in particular a Hessian metric $g$ is given by the Hessian of a globally defined function $u\in C^\infty(N,\R)$ (see, e.g., discussion in~\cite[\S 2]{hu-on-19}):
	\begin{lemma}\label{lm:hessian_simply_connected}
		Let $(N,D)$ be a simply-connected affine manifold. Then $H^1(M,\mc A)$ is trivial and any Hessian metric $g$ is the Hessian of a globally defined convex function $u\in C^\infty(N,\R)$:
		\[
		g=D^2u.
		\]
	\end{lemma}
	\begin{proof}
		Since $N$ is simply connected, the holonomy of the flat connection $D$ is trivial. Therefore, any locally defined affine section $s_0\in\Gamma(U,\mc A)$ can be uniquely extended to a global section $s\in\Gamma(M,\mc A)$. This proves that any \v{C}ech 1-cocycle in $C^1(N,\mc A)$ is a boundary so that $H^1(N,\mc A)$ is trivial.

		Let $\{U_\alpha\}$ be an open cover of $N$ and denote by $\{u_\alpha\}$ the locally convex functions defining the Hessian metric $g$. Then, since the cocycle $\{u_\alpha-u_\beta\}\in C^1(N,\mc A)$ is trivial, we can find affine functions $\{a_\alpha\}$ such that $a_\alpha-a_\beta=u_\alpha-u_\beta$. Therefore $u:=u_\alpha-a_\alpha$ is a globally defined convex function such that $D^2u=g$, as required.
	\end{proof}

	\subsection{K\"ahler Toric manifolds}\label{ss:toric_kahler}
	\subsubsection*{Symplectic data}
	Let $\T^k\simeq (S^1)^k$ be a real $k$-dimensional torus acting on a smooth manifold $M$, $\dim_\R M=2n$. In what follows we do not assume that $M$ is compact.	Denote by $\mf t\simeq \R^k$ the Lie algebra of $\T^k$. For a vector $v\in \mf t$ we set $X_v\in C^\infty(M,TM)$ to be the fundamental vector field induced by the infinitesimal action generated by $v$.

	Assume that $M$ is equipped with a symplectic form $\omega\in C^\infty(M,\Lambda^2(M))$. Recall that the action $\T^k\times M\to M$ is \textit{Hamiltonian} if there exists a moment map
	\begin{equation}\label{eq:def-moment-map}
	\bm{\mu}\colon M\to \mf t^*
	\end{equation}
	such that for every $v\in \mf t$ we have
	\begin{equation}\label{eq:def-moment-map-sign}
	\omega(X_v,\cdot)=-d\la \bm\mu,v\ra.
	\end{equation}
	It follows from the Cartan's magic formula that a Hamiltonian action of $\T^k$ preserves $\omega$, so $\T^k$ acts by symplectomorphisms.

	Let us collect some elementary properties of the moment map on a symplectic manifold.

	\begin{proposition}\label{prop:moment_properties}
		Let $(M,\omega)$ be a symplectic manifold with a Hamiltonian action of a torus $\T^k$ and the moment map $\bm{\mu}\colon M\to\mf t^*$. Then
		\begin{enumerate}
			\item $\bm{\mu}$ is uniquely determined up to an addition of a constant element in $\mf t^*$.
			\item $\bm{\mu}$ is invariant under the action of $\T^k$, or equivalently
			\[
			\omega(X_v,X_w)=0
			\]
			for any $v,w\in \mf t$.
		\end{enumerate}
	\end{proposition}
	\begin{proof}
		Clearly each of the functions $\la\bm\mu,v\ra$ is defined up to an additive constant, so $\bm\mu$ is determined up to an addition of a constant element in $\mf t^*$. Next, since $X_v$ and $X_w$ commute and preserve $\omega$, we know that $\mc L_{X_v}i_{X_w}\omega= d(\omega(X_v,X_w))=0$, so $\omega(X_v,X_w)$ is constant along the orbit of $X_v$. Since the closure of the orbit of $X_v$ is a compact submanifold, function $\la\bm\mu,w\ra$ has a critical point on this submanifold, and at such a point $\omega(X_v,X_w)=i_{X_v}d\la\bm \mu,w\ra=0$, implying that $\omega(X_v,X_w)$ is identically zero.
	\end{proof}

	In what follows we will be mainly concerned with the open dense set of points $M_0\subset M$ where the action $\T^k\times M\to M$ is free:
	\[
	M_0:=\{x\in M\ |\ \mathrm{Stab}_{\T^k}(x)=\{\mathrm{id}\}  \}
	\]
	Since the compact group $\T^k$ acts freely on $M_0$, the quotient space
	\[
	N_0=M_0/\T^k
	\]
	is naturally a smooth manifold. Denote by
	\begin{equation}\label{eq:pi_def}
	\bm{\pi}\colon M_0\to N_0
	\end{equation}
	the natural projection on the orbit space. Since $\bm{\mu}$ is constant on the orbits of $\T^k$, the moment map descends to a map
	\[
	\til{\bm\mu}\colon N_0\to \mf t^* \mbox{ such that } \bm\mu=\til{\bm{\mu}}\circ\bm{\pi}.
	\]
	Furthermore, since $\omega$ is nondegenerate and the action of $\T^k$ on $N_0$ is free, the differential $d\pmb\mu\in C^\infty(M_0,\mathrm{Hom}(TM_0,\mf t^*))$ has the maximal rank $k$.

	In this paper we will deal with a very special kind of Hamiltonian torus action~--- action of a \textit{half-dimensional} torus.

	\begin{definition}[Symplectic toric manifold]
		A \textit{symplectic toric manifold} is a connected symplectic manifold $(M,\omega)$ of dimension $\dim_\R M=2n$ with an effective Hamiltonian action of a torus $\T^n$.
	\end{definition}

	On the free locus $M_0\subset M $ of a symplectic toric manifold $(M,\omega,\T^n)$ the differential of the moment map $\pmb\mu\big|_{M_0}$ vanishes on the tangent space to the $\T^n$-orbits and has the maximal rank $n$. In particular
	\[
	\til{\bm\mu}\colon N_0\to \mf t^*
	\]
	is an open immersion.

	\begin{proposition}\label{prop:affine}
		Let $(M_0,\omega,\T^n)$ be a symplectic toric manifold with the free torus action $\T^n\times M_0\to M_0$. Then the orbit space $N_0=M_0/T^n$ inherits a unique affine structure $(N_0,D)$ such that the functions $\la \til{\bm{\mu}},v\ra, v\in\mf t$ are affine, and $\til{\bm\mu}\colon N_0\to \mf t^*$ is the development map.
	\end{proposition}
	\begin{proof}
		Since the map $\til{\bm\mu}\colon N_0\to \mf t^*$ is an open immersion, the natural affine structure on $\mf t^*\simeq \R^n$ can be pulled back to $N_0$. By definition, $\la \til{\bm{\mu}},v\ra$ are affine functions for this structure, and being a local coordinate system, these functions uniquely determine the affine structure on $N_0$.
	\end{proof}

	We now describe a local coordinate system on $M_0$. For the notational convenience, let us introduce a basis $e^1,\dots,e^n$ the weight lattice of $\mf t$, denote by $e_1,\dots,e_n\in\mf t^*$ the dual basis. For a vector $v\in\mf t$ we will denote by $\{v_i\}$ its coordinates relative to $\{e^i\}$, and respectively for a covector $\xi\in\mf t^*$, we denote by $\{\xi^i\}$ its components relative to $\{e_i\}$. Let $x^i:=\la\bm\mu,e^i\ra\in C^\infty(M,\R)$ be the components of the moment map. Functions $\{x^i\}_{i=1}^n$ can be augmented by locally defined \textit{angular} functions $\theta_i\colon M\to \R/2\pi \Z$, $i=1,\dots n$ to provide a local coordinate system on $M_0$.

	\begin{lemma}\label{lm:omega_coordinates}
		On a symplectic toric manifold $(M_0,\omega,\T^n)$ with a free Hamiltonian action of $\T^n$ the moment map $\bm{\mu}=(x^1,\dots x^n)\colon M_0\to \mf t^*\simeq \R^n$ can be augmented by local $\T^n$-invariant functions $\{\theta_1,\dots,\theta_n\}$ such that $X_{e^j}\cdot\theta_i=\delta_i^j$,
		providing a Darboux coordinate system for $\omega$:
		\begin{equation}\label{eq:omega_coordinates}
		\omega=\sum_{i=1}^n dx^i\wedge d\theta_i.
		\end{equation}
	\end{lemma}

	\begin{proof}
		For $x\in M_0$ and let $\T^n. x\simeq \T^n$ be the orbit of a point $x$. Let $V^n\subset M_0$ be a local Lagrangian submanifold passing through $x\in M_0$ transversally to $\T^n.x$. Transporting $V^n$ by the action of $\T^n$, we get an equivariant identification between a neighbourhood $U(\T^n.x)$ of $\T^n. x$ and $\T^n\times V^n$:
		\begin{equation}\label{eq:coords_id}
			U(\T^n.x)\simeq \T^n\times V^n
		\end{equation}
		Denote by $\{\theta_i\}_{i=1}^n$ the (multivalued angular) coordinates on $\T^n$ such that $d\theta_i(X_{e^j})=\delta_i^j$ and extend them to functions on $U(\T^n.x)$ via~\eqref{eq:coords_id}. Then in the basis $\{dx^1,\dots,dx^n,d\theta_1,\dots d\theta_n\}$ the 2-vector $\omega^{-1}$ takes the standard form:
		\[
		\omega^{-1}(dx^i,dx^j)=\omega^{-1}(d\theta_i,d\theta_j)=0,\quad \omega^{-1}(d\theta_i,dx^j)=\delta_i^j,
		\]
		or equivalently, $\{dx^1,\dots,dx^n,d\theta_1,\dots d\theta_n\}$ is the Darboux coordinate system for $\omega$ on $U(\T^n.x)$:
		\begin{equation}
			\omega=\sum_{i=1}^n dx^i\wedge d\theta_i=\la d\bm{\mu}\wedge d\bm{\theta}\ra,
		\end{equation}
		where $d\bm{\theta} \colon T U(\T^n.x)\to \mf t$ is a (locally defined) flat principal $\T^n$-connection defined by the identification~\eqref{eq:coords_id}.
	\end{proof}

	\begin{remark}
		From now on, for a symplectic toric manifold $(M_0,\omega,\T^n)$ with a free torus action, we always refer to $\{x^1,\dots,x^n\}$ as the local coordinate system on the orbit space $N_0=M_0/\T^n$. In particular, in a neighborhood of any point $x\in M_0$, forms $\{dx^1,\dots dx^n\}$ generate over $C^\infty(N_0,\R)$ the algebra of the \textit{basic} differential forms of the projection $\bm{\pi}\colon M_0\to N_0$.
	\end{remark}

	\subsubsection*{K\"ahler data}
	Further on we assume that symplectic form $\omega$ is a \textit{K\"ahler form} of a complex K\"ahler manifold $(M,g,J)$, that is there exist an integrable complex structure $J\colon TM\to TM$, $J^2=-\mathrm{Id}$ and a compatible Riemannian metric $g$ such that $\omega(\cdot,\cdot)=g(J\cdot,\cdot)$. We denote by $J^*\colon T^*M\to T^*M$ the dual operator induced by $J$: so that $J^*\alpha (v)=\alpha(Jv)$ for $\alpha\in T^*M$, $v\in TM$.

	\begin{definition}[K\"ahler toric manifold]
	A \textit{K\"ahler toric manifold} is a connected K\"ahler manifold $(M,g,J)$, of dimension $\dim_\R M=2n$ with $\omega=gJ$ and an effective, biholomorphic, Hamiltonian action of an $n$-dimensional torus $\T^n$.
	\end{definition}

	\begin{ex}
		The basic examples of toric K\"ahler manifolds are
		\[
		(\C^n,\omega_{\mathrm{st}}),\quad \omega_{\mathrm{st}}=\i \sum d z^i\wedge d \bar{z}^i
		\]
		and
		\[
		((\C^*)^n,\omega_{\mathrm{log}}),\quad \omega_{\mathrm{log}}=\i \sum \frac{d z^i}{z_i}\wedge \frac{d \bar{z}^i}{\bar{z_i}},
		\]
		where $\{z_i\}$ are complex coordinates on $\C^n$ (respectively on $(\C^*)^n$). Both these examples give rise to \textit{complete} Riemannian metrics on the underlying manifolds. Furthermore, the infinitesimal action of $\mf t_\C=\mf t\otimes \C$ can be integrated to the action of the algebraic torus $(\C^*)^n\supset\T^n$.

		An example of a slightly different flavor is given by a polydisc with the hyperbolic Poincare metric
		\[
		(\D^n, \omega_{\mathrm{hyp}}),\quad \omega_{\mathrm{hyp}}=\i \sum \frac{dz^i\wedge d\bar{z}^i}{(1-|z^i|^2)^2}.
		\]
		While this example still defines a complete metric on the underlying Riemannian manifold, the infinitesimal action of $\mf t_\C=\mf t\otimes \C$ cannot be integrated to the action of $(\C^*)^n$, since the vector fields $JX_v$ for nonzero $v\in \mf t$ are not complete, that is the flows of $JX_v$ are only defined locally.
	\end{ex}
	\begin{remark}
		There are two parallel approaches to the toric geometry in symplectic and complex-algebraic categories. Many K\"ahler manifolds can be considered both as symplectic and as algebraic, making the notion of \emph{toric K\"ahler manifold} ambiguous (see \cite[\S 2.5]{ci-20}). In this paper we take the `symplectic' approach to toric K\"ahler geometry, and do not assume that the underlying complex manifolds $(M,J)$ are algebraic. In particular, we do not assume that the action of $\T^n\simeq (S^1)^n$ can be complexified to the action of $(\C^*)^n$.
	\end{remark}

	Next we describe the data of a toric K\"ahler manifold in terms of a potential function in appropriate local coordinates in a neighbourhood of a free orbit $\T^n.x\subset M_0$ . Let $(M,g,J,\T^n)$ be a toric K\"ahler manifold.
	The local coordinate system of Lemma~\ref{lm:omega_coordinates} in a neighbourhood of a point $x\in M_0$ depends on an identification of a neighbourhood $U(\T^n.x)$ of the orbit $\T^n.x$ with the product $\T^n\times V$. On a complex manifold there is a preferred biholomorphic identification of $U(\T^n.x)$ with the neighbourhood of the \textit{real} torus $\T^n\simeq U(1)^n\subset (\C^*)^n\simeq \T^n\times \R^n$ in the complex torus $\T_\C^n=(\C^*)^n$. Let us denote by $y_i+\i \theta_i$ the holomorphic coordinates on $(\C^*)^n$ such that $\theta_i\in \R/2\pi\Z$ are angular coordinates on $\T^n$.

	The advantage of the corresponding angular coordinates is that the functions $\{\theta_i\}$ are $J$-pluriharmonic, i.e., the 1-forms $J^*d\theta_i:=d\theta_i(J\cdot)=dy_i$ are closed:
	\begin{equation}\label{eq:pluriharmonic_theta}
	d(J^*d\theta_i)=0,\quad \mbox{locally } J^*d\theta_i=dy_i.
	\end{equation}
	Furthermore, forms $dy_i$ are clearly \textit{basic} with respect to the projection $\bm{\pi}\colon M_0\to N_0$. Therefore, on $U(\T^n.x)$ we can write
	\begin{equation}\label{eq:def_G}
	dy_i=J^*d\theta_i=\sum_{j=1}^n\bm G_{ij}dx^j,
	\end{equation}
	where $\bm G_{ij}$ is a matrix-valued function on $N_0$. We can equivalently rewrite~\eqref{eq:def_G} as
	\[
	J^*dx^i=-\sum_{j=1}^n (\bm G^{-1})^{ij}d\theta_j,
	\]
	which gives us an invariant globally-defined expression for $\bm G$ through its inverse:
	\[
	(\bm G^{-1})^{ij}=-\la dx^i,JX_{e^j}\ra=g(X_{e^i},X_{e^j}).
	\]
	In particular, $\bm G$ can be thought of as a positive definite $\mathrm{Sym}^2(\mf t)$-valued function defined globally on $N_0$. Then the condition~\eqref{eq:pluriharmonic_theta} can be restated as an integrability condition for $\bm{G}$ in coordinates $\{x^1,\dots,x^n\}$:
	\[
	\frac{\pd \bm G_{ij}}{\pd x_k}=\frac{\pd \bm G_{ik}}{\pd x_j}.
	\]
	Since $\bm{G}$ is symmetric, it implies that the 1-form $y_idx^i$ is closed, thus there is a local function $u(x)$ such that $y_i=\frac{\pd u}{\pd x_i}$ and
	\[
	\bm G_{ij}=\frac{\pd^2 u}{\pd x_i \pd x_j}
	\]
	Function $u(x)$ is called a \textit{symplectic potential} for the K\"ahler structure on a symplectic toric manifold $(M_0,\omega,\T^n)$. We can collect the above observations in the following lemma.

	\begin{lemma}\label{lm:kahler_coordinates}
		Let $(M_0,g,J,\T^n)$ be a symplectic toric manifold with a free action of $\T^n$, the orbit space $N_0=M_0/\T^n$ and a moment map $\bm{\mu}\colon M_0\to \mf t^*$. Then for any $\omega$-compatible K\"ahler data $(g,J)$ on $M_0$, in a neighbourhood of any point $x\in M_0$, the moment coordinates $\{x^1,\dots,x^n\}$ can be augmented by angular coordinates $\{\theta_1,\dots,\theta_n\}$ such that
		\begin{equation}
		\omega=\sum_{i=1}^n dx^i\wedge d\theta_i.
		\end{equation}
		Furthermore, in a neighbourhood of $\bm \mu(x)\in N_0$ there exists a function $u(x)$ such that
		\begin{equation}\label{eq:kahler_coordinates}
		\begin{split}
		g=\sum_{i,j=1}^n \bm G_{ij}dx^idx^j&+\sum_{i,j=1}^n \bm (\bm G^{-1})^{ij}d\theta_id\theta_j\\
		J^*d\theta_i=\sum_{j=1}^n\bm G_{ij}dx^j,&\qquad J^*dx^i=-\sum_{j=1^n}(\bm G^{-1})^{ij}d\theta_j,
		\end{split}
		\end{equation}
		where $\bm G_{ij}=\frac{\pd^2 u}{\pd x^i\pd x^j}$.

		Conversely, any function $u(x)$ with a positive Hessian on an open subset $U\subset \mf t^*$ defines a local K\"ahler structure on the product $U\times \T^n$, $U\subset \mf t^*$ via the identities~\eqref{eq:kahler_coordinates}.
	\end{lemma}

	\begin{corollary}\label{cor:hessian}
		Let $\bm{\pi}\colon M_0\to N_0=M_0/\T^n$ be the projection onto the orbit space of a toric K\"ahler manifold $(M_0,g,J,\T^n)$ with a free action of $\T^n$. Then $N_0$ equipped with its natural affine structure of Proposition~\ref{prop:affine} admits a natural Hessian structure $\bm{\pi}_*g$ such that the natural projection
		\[
		(M_0,g)\to (N_0, \bm{\pi}_*g)
		\]
		is a Riemannian submersion. 
	\end{corollary}

	\begin{proof}
		Let $(M_0,g,J,\T^n)$ be a toric K\"ahler manifold with a free action of $\T^n$. By Lemma~\ref{lm:kahler_coordinates}, in a neighbourhood of a point $x\in M_0$, the metric $g$ is given by
		\[
		g=\sum_{i,j=1}^n \bm G_{ij}dx^idx^j+\sum_{i,j=1}^n \bm (\bm G^{-1})^{ij}d\theta_id\theta_j,
		\]
		where
		\[\bm G_{ij}=\frac{\pd^2u}{\pd x^i\pd x^j}\] for a locally defined function $u\colon U\subset N_0\to \R$. Thus, identifying $TN_0$ with the horizontal space $\mathrm{ker}(d\bm{\pi})^{\perp}\subset TM_0$ for the natural submersion $\bm{\pi}\colon M_0\to N_0$, we obtain a Hessian metric on $N_0$:
		$\bm{\pi}_*g:=\sum_{i,j=1}^n \bm G_{ij}dx^idx^j$.
	\end{proof}

	\begin{remark}
		If the orbit space $N_0$ is simply connected, then the angular coordinates $\{\theta_i\}$ and the potential function $u(x)$ can be defined globally. Indeed, by the integrability of the complex structure $J$, the Lagrangian distribution $J\mathrm{span}(X_{e^1},\dots X_{e^n})$ on $M_0$ is integrable. Since $N_0=M_0/T^n$ is simply connected, we can find a global section
		\[
		{\bm\iota}\colon N_0\to M_0
		\]
		tangent to this distribution. Setting $\theta_i\big|_{\bm\iota(N_0)}=0$ we can uniquely extend the $S^1$-valued functions $\theta_i$ to $\T^n$-equivariant functions on $M_0$ such that $\theta_i(X_{e^j})=\delta_i^j$. Now we can invoke Lemma~\ref{lm:hessian_simply_connected} and find a global potential $u\colon N_0\to \R$ for the metric ${\bm\pi}_*g$.
	\end{remark}

	\subsection{Toric K\"ahler-Ricci solitons}\label{ss:toric_krs}

	Recall that a connected Riemannian manifold $(M,g)$ equipped with a function $f\in C^\infty(M,\R)$ is a \textit{gradient Ricci soliton} if
	\begin{equation}\label{eq:RS}
	\Ric+\n_g^2f=\lambda g,\quad \lambda\in \{-1,0,+1\},
	\end{equation}
	where $\Ric$ is the Ricci tensor associated with $g$ and $\n_g$ is the Levi-Civita connection. A gradient Ricci soliton is called respectively \textit{expanding, steady or shrinking} depending on whether $\lambda=-1$, $\lambda=0$ or $\lambda=1$. In this paper we will be only concerned  with the gradient steady Ricci solitons.

	\begin{remark}
		As we mentioned in the introduction, it easily follows from the maximum principle that any \textit{compact} gradient steady soliton $(M,g,f)$ is necessary Ricci flat with $\Ric=0$ and we can take $f\equiv 0$. Since we are interested in the \textit{genuine} gradient solitons, i.e., with $\Ric\neq 0$, we are forced to consider non-compact manifolds. A geometrically important and natural setting in this case is to consider \textit{complete} gradient steady solitons $(M,g,f)$, and this will be our standing assumption through the most of the paper.
	\end{remark}

	If the Riemannian structure $(M,g)$ underlying a gradient Ricci soliton $(M,g,f)$ is a part of a K\"ahler structure $(M,g,J)$, then we say that $(M,g,J,f)$ is a gradient \textit{K\"ahler-Ricci} soliton, or KRS for short. It is well-known that on a K\"ahler background equation~\eqref{eq:RS} can be rewritten as
	\begin{equation}\label{eq:KRS}
	\rho_{\omega}+\i\pd\bar{\pd}f=2\lambda\omega.
	\end{equation}
	Here $\rho_\omega$ is the Ricci form~--- the curvature of the anticanonical bundle $-K_M$, which can be computed locally as
	\[
	\rho_\omega = \i\pd\bar{\pd}\log|\sigma|^2_g,
	\]
	where $\sigma$ is a local holomorphic $(n,0)$-form.
	The term $\i\pd\bar{\pd}f$ can be expressed as $\frac{1}{2}\mc L_{X}\omega$, where $X=\n_g f$ is a gradient vector field. Furthermore, vector field $JX$ is necessarily Killing and holomorphic, see, e.g., Bryant~\cite{br-08}.

	Now we impose the toric symmetry ansatz on a K\"ahler manifold $(M,g,J)$, which we will keep throughout the paper.
	\begin{definition}\label{def:toric_soliton}
		A toric K\"ahler manifold $(M,g,J ,\T^n)$ is a \textit{gradient steady toric K\"ahler-Ricci soliton} if there exists a function $f\in C^\infty(M,\R)$ such that the vector field $\n_g f-\i J\n_g f$ is holomorphic and
			\begin{equation}\label{eq:KRS_steady}
			\rho_{\omega}+\i\pd\bar{\pd}f=0.
			\end{equation}
	\end{definition}
	
	Vector field $\n_g f-\i J\n_g f$ being holomorphic is equivalent to the vanishing $(\n^2_g f)^{2,0}=0$ of the $(2,0)$-type part of the Hessian of $f$. Since we only consider gradient steady Ricci solitons in this paper, we often will not mention it explicitly, and simply call the manifolds of Definition~\ref{def:toric_soliton} \textit{toric K\"ahler-Ricci solitons (KRS).} Let $(M,g,J ,\T^n,f)$ be a toric KRS. Since $(g,J )$ are $\T^n$-invariant, by averaging $f$ with respect to the action of $\T^n$, we can assume that $f$ is also $\T^n$ invariant and still solves~\eqref{eq:KRS_steady}.
	\begin{proposition}\label{prop:toric_KRS_invariant}
		Let $(M,g,J ,\T^n,f)$ be a gradient steady toric K\"ahler-Ricci soliton with an invariant soliton function $f\in C^\infty(M,\R)$. Then there exists $v\in \mf t$ and constant $c$ such that
		\begin{equation}\label{eq:soliton_vf}
		J\n_g f=X_v,\quad f=\la \bm{\mu},v\ra+c.
		\end{equation}
		Furthermore, there exist $\xi\in \mf t^*$ such that in the coordinates $\{x^1,\dots,x^n,\theta_1,\dots,\theta_n\}$ of Lemma~\ref{lm:kahler_coordinates}
		\begin{equation}\label{eq:soliton_eq_invariant}
		d(\log\det \mb G_{ij}+\la \bm{\mu},v\ra)=\xi^iJ^*\left(d\theta_i\right).
		\end{equation}
		Conversely, if a K\"ahler toric manifold $(M,g,J ,
		T^n)$ satisfies~\eqref{eq:soliton_eq_invariant} on $M_0\subset M$ for $f$ given by~\eqref{eq:soliton_vf}, then $(M,g,J ,\T^n,f)$ is a gradient steady toric KRS.
	\end{proposition}
	\begin{proof}
		Let $M_0\subset M$ be an open dense set, where the action $\T^n\times M_0\to M_0$ is free. Denote by $N_0=M_0/\T^n$ the orbit space, and by $\til {\bm \mu}\colon N_0\to \mf t^*$ the moment map descended to $N_0$. Recall that $x^i=\la\bm{\mu},e^i\ra$ provide a local coordinate system on $N_0$. We will prove identity~\eqref{eq:soliton_vf} on $M_0$, and by continuity it must hold on the entire $M$.

		Since $f$ is invariant under the action of $\T^n$, we have that $df$ is basic, i.e.,
		\[
		df\in \mathrm{span}_{C^\infty(N_0,\R)}\{dx^1,\dots dx^n\}
		\]
		Since $J$ establishes an isomorphism between the vertical subspaces $\mf t.x\subset TM_0$ and the horizontal subspace $\mathrm{span}\{g^{-1}dx^i\}$, we have that
		\[
		J\n_g f\in \mathrm{span}_{C^\infty(N_0,\R)}\{X_{e^1},\dots X_{e^n}\},\quad J\n_g f=v_i(x)X_{e^i}.
		\]
		On the other hand, $X_{e^i}$ is the real part of a holomorphic vector field by the toric assumption, and $J\nabla_g f$ is the real part of a holomorphic vector field by the soliton assumption. Thus the real-valued functions $v_i(x)$ are holomorphic, and hence constant, proving that
		\[
		J\nabla_gf=X_v
		\]
		for some $v\in \mf t$. Equivalently, we can rewrite this identity as $df=-\omega(X_v,\cdot)$. Since by the definition of the moment map $\omega(X_v,\cdot)=-d\la\bm{\mu},v\ra$, we conclude that up to an additive constant
		\[
		f=\la {\bm\mu},v\ra=x^iv_i,
		\]
		where $v_i$ are the components of $v$ relative to the fixed basis of $\mf t$.

		Next, since $(M,g,J ,f)$ is a gradient steady toric K\"ahler-Ricci soliton, we have
		\[
		\rho_\omega+\i\pd\bar{\pd}f=0.
		\]

		\begin{lemma}\label{lm:ricci_form_coordinates}
			The Ricci form on $M_0$ is given by
			\begin{equation}\label{eq:ricci_form_coordinates}
			\rho_\omega=\i\pd\bar{\pd}\log\det\bm G_{ij},
			\end{equation}
			where $\bm G_{ij}\in \Sym^2\mf t$ is as in the proof of Lemma~\ref{lm:kahler_coordinates}, and the determinant is computed relative to the local affine coordinates $\{x^i\}$ on $N_0$.
		\end{lemma}

		\begin{proof}[Proof of the lemma.]
			Both sides of~\eqref{eq:ricci_form_coordinates} are globally defined on $M_0$, thus it suffices to prove it locally. Consider a neighbourhood of a free orbit $U(\T^n.x)$ and identify it with a neighborhood of real torus $U(1)^n\subset (\C^*)^n$. This identification provides local holomorphic coordinates
			\[
			\{y_i+\i \theta_i\}
			\]
			as in the proof of Lemma~\ref{lm:kahler_coordinates}.

		This coordinates determine a distinguished holomorphic volume form
		\[
		\sigma=(y_1+\i \theta_1)\wedge\dots\wedge(y_n+\i \theta_n).
		\]
		Given the explicit form of the metric $g$ in the coordinates $\{x^1,\dots,x^n,\theta_1,\dots,\theta_n\}$ we find
		\[
		|\sigma|^2_g=\frac{\i^{n^2}\sigma\wedge \bar{\sigma}}{dV_g}=c_n \frac{(dy_1\wedge\dots \wedge dy_n)\wedge (d\theta_1\wedge\dots \wedge d\theta_n)}{(dx^1\wedge\dots \wedge dx^n)\wedge (d\theta_1\wedge\dots \wedge d\theta_n)}=c_n\frac{dy_1\wedge\dots \wedge dy_n}{dx^1\wedge\dots \wedge dx^n},
		\]
		where $dV_g$ is the Riemannian volume form on $M$, $c_n$ is a constant, and the last fraction is considered as a ratio of two sections of the determinant bundle $\Lambda^n(T^*N_0)$ of the orbit space. It remains to observe that according to the proof of Lemma~\ref{lm:kahler_coordinates}
		\[
		dy_i=\bm{G}_{ij}dx^j,
		\]
		hence $|\sigma|^2_g=c_n\det{\bm G}$ and
		\[
		\rho_\omega=\i\pd\bar{\pd}\log\det\bm{G}.
		\]
		Lemma~\ref{lm:ricci_form_coordinates} is proved.
		\end{proof}

		Since $\rho_\omega=\i\pd\bar{\pd}\log\det \bm{G}_{ij}$, the soliton equation~\eqref{eq:KRS_steady} implies that the real-valued function $\Phi:=\log\det \bm{G}_{ij}+\la\bm{\mu},v\ra$ is pluriharmonic:
		\[
		\i\pd\bar{\pd}\Phi=0.
		\]
		On the other hand, being $\T^n$-invariant, function $\Phi$ can be expressed as a function of the real coordinates $y_i$. Since $y_i$ are themselves real pluriharmonic, $\Phi$ must be a linear function of $y_i$:
		\[
		\Phi=y_i\xi^i+c,
		\]
		for some $\xi\in \mf t^*$, furthermore, such $\xi$ is unique. Therefore we have
		\[
		\log\det \bm{G}_{ij}+\la\bm{\mu},v\ra=y_i\xi^i+c.
		\]
		Using the local form of $g$ and $J$ given by~\eqref{eq:kahler_coordinates} we can equivalently express this identity as
		\[
		d(\log\det \bm{G}_{ij}+\la\bm{\mu},v\ra)=\xi^iJ^*\left(d\theta_i\right).
		\]
		as claimed.

		Conversely, if a K\"ahler toric manifold $(M,g,J ,\T^n)$ admits $v\in\mf t$ and $\xi\in \mf t^*$ such that~\eqref{eq:soliton_eq_invariant} holds, then $\rho_\omega+\i\pd\bar{\pd}f=0$ for $f=\la\bm{\mu}, v\ra$, so that $(M,g,J )$ is a gradient steady K\"ahler-Ricci soliton.
	\end{proof}

	\begin{corollary}[Local weighted Monge-Amp\`ere equation]\label{cor:MA_eq}
		If $(M,g,J ,\T^n,f)$ is a gradient steady toric K\"ahler-Ricci soliton with an invariant soliton function $f\in C^\infty(M,\R)$, and $u(x)\colon U_\alpha\subset N_0\to\R$ is a locally defined symplectic potential on $N_0$ so that:
		\[
		\bm G_{ij}=\frac{\pd^2 u}{\pd x^i\pd x^j}
		\]
		\[
		y_i=\frac{\pd u}{\pd x^i},
		\]
		then function $u(x)$ solves the Monge-Amp\`ere equation
		\begin{equation}\label{eq:MA}
		\log\det \frac{\pd^2 u}{\pd x^i\pd x^j}=-v_ix^i+\xi_i\frac{\pd u}{\pd x^i}+c.
		\end{equation}
		where $v\in\mf t$ and $\xi\in \mf t^*$ are as in Proposition~\ref{prop:toric_KRS_invariant} and $c$ is a constant.
	\end{corollary}

	\begin{remark}
		A symplectic potential $u(x)$ is defined up to the addition of an affine summand $(v_0)_ix^i+c_0$. If $\til u(x)=u(x)+(v_0)_ix^i+c_0$ is a new local potential, then $\til u(x)$ solves the same equation~\eqref{eq:MA} as $u(x)$ but with a new constant $\til c=c-\la\xi,v_0\ra$.
	\end{remark}

	We see that, if $U_\alpha$ is an open cover of $N_0$, and $u_{\alpha}$ are local symplectic potentials, then the differences of the locally defined functions
	\[
	F_\alpha:=\log\det \frac{\pd^2 u}{\pd x^i\pd x^j}+v_ix^i-\xi_i\frac{\pd u}{\pd x^i}
	\]
	yield a cocycle $\{F_\alpha-F_\beta\}\in C^1(N_0,\R)$. In the special case when $N_0$ is simply connected, this cocycle is a coboundary and by Lemma~\ref{lm:hessian_simply_connected} the Hessian structure $(N_0, \bm G_{ij})$ is given by a globally defined symplectic potential $u\colon N_0\to \R$. Therefore we can formulate a refinement of Corollary~\ref{cor:MA_eq}

	\begin{corollary}[Global weighted Monge-Amp\`ere equation]\label{cor:MA_eq_simply_connected}
		Let $(M,g,J ,\T^n,f)$ is a gradient steady toric K\"ahler-Ricci soliton with an invariant soliton function $f\in C^\infty(M,\R)$. Then there exists a convex function $u(x)$ on the universal cover $\til{N_0}$ of $N_0$, $v\in \mf t$, $\xi \in \mf t^*$ and $c\in\R$ such that
				\[
				\bm G_{ij}=\frac{\pd^2 u}{\pd x^i\pd x^j}
				\]
				\[
				y_i=\frac{\pd u}{\pd x^i},
				\]
		and $u(x)$ solves the Monge-Amp\`ere equation on $\til{N_0}$
		\begin{equation}\label{eq:MA_global}
		\log\det \frac{\pd^2 u}{\pd x^i\pd x^j}=-v_px^p+\frac{\pd u}{\pd x^q}\xi^q+c.
		\end{equation}
	\end{corollary}
    \begin{proof}
        Let $\bm\pi\colon M_0\to N_0$ be projection onto the orbit space. $N_0$ is equipped with the moment map $\til{\bm\mu}\colon N_0\to \mf t^*$ and the Hessian structure of Corollary~\ref{cor:hessian}. By Lemma~\ref{lm:hessian_simply_connected}, on the universal cover $\til N\to N$ the metric $\bm\pi_*g$ can be expressed as $g=D^2u(x)$ for some $u\in C^\infty(\til N_0,\R)$. By the above Corollary~\ref{cor:MA_eq}, in a neighborhood $U_\alpha$ of any point on $\til N$, function $u(x)$ satisfies
        \[
        \log\det \frac{\pd^2 u}{\pd x^i\pd x^j}=-v_px^p+\frac{\pd u}{\pd x^q}\xi^q+c_\alpha.
        \]
        It remains to note that vector and covector $v$ and $\xi$ being locally constant, i.e. $D$-parallel, must be globally constant since $\til N$ is simply connected. As all the terms of the above equation except possibly for $c_\alpha$ are globally defined, it follows that $c_\alpha=c$ is also globally defined.
    \end{proof}

	Our goal in the rest of the paper will be to derive a priori estimates for the solutions of the equation~\eqref{eq:MA_global} and to use those to classify \textit{complete} solutions.

	\section{Geometry of the Bakry-\'Emery Ricci tensor}\label{s:bakry-emery}

    In this section we will make a short detour into the geometry of metric measure spaces and Bakry-\'Emery Ricci tensors $\Ric_\phi$. This geometric setting is clearly relevant for the study of Ricci solitons, since by definition a Ricci soliton $(M,g,\phi)$ has a constant Bakry-\'Emery tensor $\Ric_\phi=\lambda g$. However, it is worth pointing out, that it is not the toric KRS itself which we will consider through the lens of Bakry-\'Emery geometry, but its orbit space $N_0=M_0/\T^n$ (see Sections~\ref{s:ma} and \ref{s:main}). Furthermore the relevant weight function $\phi$ will not be the soliton potential $f$.

	Let $(N,g,e^{-\phi}dV_g)$ be a Riemannian manifold with a smooth positive measure $e^{-\phi}dV_g$ defined by $\phi\in C^\infty(N,\R)$. Such data is called a (smooth) \textit{metric measure space}. To a metric measure space we can associate several natural geometric objects.

	\begin{definition}
		The \textit{Bakry-\'Emery Ricci tensor} associated to $(N,g,e^{-\phi}dV_g)$ is defined by
		\[
		\Ric_\phi:=\Ric+\n_g^2\phi,
		\]
		where $\Ric$ is the Ricci tensor of $(M,g)$. The \textit{weighted Laplacian} of $(N,g,e^{-\phi}dV_g)$ is the operator $\Delta_\phi\colon C^\infty(N,\R)\to C^\infty(N,\R)$ given by
		\[
		\Delta_\phi F:=\Delta F-\la \n_g\phi,\n_g F\ra_g
		\]
	\end{definition}

	\begin{remark}
		Tensor $\Ric_\phi$ is also known as the $\infty$-Bakry-\'Emery tensor. For certain geometric applications a more general $N$-Bakry-\'Emery tensor $\Ric_{\phi,N}:=\Ric+\n_g^2\phi-\frac{1}{N}d\phi\otimes d\phi$ happens to be more useful, see, e.g.,~\cite{lo-03}.
	\end{remark}

	Fixing a base point $p_0$ on a smooth metric measure space $(N,g,e^{-\phi}dV_g)$ as above let $r(x)=d(p_0,x)$ be the distance from $p_0\in N$. Denote by $B_r(p_0)$ an open ball of radius $r$ around $p_0$ and by $\pd B_r(p_0)$ its boundary. Define
	\begin{equation}\label{eq:mean_curvature}
	m_\phi(x)=\Delta_\phi r=\Delta r - \la \n_g \phi, \n_g r\ra_g
	\end{equation}
	the weighted mean curvature of the geodesic sphere centered at $p_0$. We have the following fundamental comparison result for the weighted mean curvature due to Wei and Wylie~\cite{we-wy-09}. This result can be though of as a generalization of the Laplacian comparison theorem to the metric measure spaces.

	\begin{theorem}[Mean Curvature Comparison I,~\cite{we-wy-09}]\label{th:mean_curvature_I}
		Consider a metric measure space with $\Ric_\phi\geq 0$ and fix $\delta>0$. Then for $x$ such that $r(x)>\delta$ we have
		\[
		m_\phi(x)\leq \sup \{m_\phi(x_0)\ |\ {x_0\in \pd B_{\delta}(p_0)}\}
		\]
	\end{theorem}
	\begin{proof}
		The proof of this theorem is short and instructive, so we include it for the reader's convenience following~\cite{we-wy-09}. By the Bochner formula
		\[
		\frac{1}{2}\Delta|\n_g u|^2=|\n_g^2u|^2+\la \n_gu,\n_g(\Delta u)\ra_g+\Ric(\n_g u,\n_g u)
		\]
		applied to $u(x)=r$ we observe that
		\[
		0=|\nabla_g^2 r|^2+\nabla_gr\cdot \Delta r+\Ric(\n_g r,\n_g r).
		\]
		Since $|\nabla_g^2 r|^2\geq 0$ and $\Ric_\phi(\n_g,\n_g)=\Ric(\n_g,\n_g)+\n_g^2\phi(\n_g r,\n_g r)\geq 0$, we conclude
		\[
		\n_g r\cdot(m_\phi(x))= \n_gr\cdot (\Delta r-\la \n_g \phi,\n_gr\ra_g)=\n_g r\cdot \Delta r-\n^2_g \phi(\n_g r,\n_g r)\leq 0
		\]
		Therefore, $m_\phi(x)$ is non-increasing along the geodesics from $p_0$, and $m_\phi(x)$ does not exceed the maximum of $m_\phi$ on $\pd B_\delta(p_0)$ as long as $r(x)>\delta$.
	\end{proof}

	While Theorem~\ref{th:mean_curvature_I} will suffice for our purposes in this paper, it provides no information on the blow-up behavior of $m_\phi(x)$ near $p_0$. It is worth noticing that~\cite{we-wy-09} contains several refinements, which allow for a more precise control of $m_\phi(x)$ under reasonable assumptions on $\phi$. For example, one has the following:

	\begin{theorem}[Mean Curvature Comparison II,~\cite{we-wy-09}]
		Assume $\Ric_\phi\geq 0$ and $|\phi|\leq C$ along a minimal geodesic segment from $p_0$ to $x$. Then
		\[
		m_\phi(x)\leq \frac{n+4C-1}{r(x)}.
		\]
	\end{theorem}

	The following Liouville-type estimate is well-known in the classical Riemannian case (when $\phi=0$). For our applications, we need to re-derive it for a smooth metric measure space with nonnegative Bakry-\'Emery tensor.

	\begin{theorem}\label{th:lioville}
		Assume that a metric measure space $(N,g,e^{-\phi}dV_g)$ has a nonnegative Bakry-\'Emery Ricci tensor $\Ric_\phi\geq 0$. Let $s(x)$ be a nonnegative smooth function defined on an open geodesic ball $B_R(p_0)\subset N$, $R>1$ such that
		\begin{equation}\label{eq:diff_ineq_s}
		\Delta_\phi s\geq s^2
		\end{equation}
		Given $\delta\in (0,1)$, there exists a constant
		\[
		C=C(\delta, m_\delta),\quad m_\delta:=\sup \{m_\phi(x)\ |\ {x\in \pd B_{\delta}(p_0)}\}
		\]
		such that function $s(x)$ satisfies inequality
		\[
		s(p_0)\leq C R^{-1}.
		\]
		In particular, if $(N,g)$ is complete, then any globally defined function $s(x)$ satisfying~\eqref{eq:diff_ineq_s} is identically zero.
	\end{theorem}
	\begin{proof}
		Let $\eta(t)\colon [0,R]\to [0,+\infty)$ be the cut-off function defined by
		\begin{equation}
		\eta(t)=\begin{cases}
		R,\quad t\leq \delta\\
		R(1-(\frac{t-\delta}{R-\delta})^2)^2, t> \delta
		\end{cases}
		\end{equation}
		We readily check that $\eta(t)$ decreases from $\eta(0)=R$ to $\eta(R)=0$, and given that $\delta\in (0,1)$, $R>1$ there exists a constant $c_0(\delta)$ such that
		\begin{enumerate}
			\item $0\leq -\eta'(t)< c_0(\delta)$;
			\item $|\eta''(t)|<c_0(\delta)$;
			\item $(\eta')^2\eta^{-1}<c_0(\delta)$.
		\end{enumerate}
		With the above cut-off function we consider the non-negative function.
		\[
		\Phi(x):=\eta(r)s(x)
		\]
		Since $s(x)$ is nonnegative and $\eta(r)$ vanishes for $r\geq R$, $\Phi(x)$ attains a local maximum at a point $p\in B_R(p_0)$, so that at $p$ we have
		\begin{equation}
		\begin{split}
		\n_g \Phi(p)&=0\\
		\Delta \Phi(p) &\leq 0.
		\end{split}
		\end{equation}
		Expanding the above identities, we find that at $p$
		\begin{equation}\label{eq:pf_1st_order}
		\eta \n_g s+\eta' s \n_g r=0\\
		\end{equation}
		and
		\begin{equation}
		\begin{split}
		0&\geq \Delta_\phi\Phi=\Delta \Phi-\la \n_g \phi, \n_g\Phi\ra_g\\&=
		s\Delta\eta +2\la \n_g\eta, \n_gs\ra_g+\eta \Delta s-\eta\la \n_g\phi, \n_gs\ra_g-s\la \n_g\phi,\n_g\eta\ra_g\\
		&=s\Delta_\phi\eta+2\la \n_g\eta, \n_gs\ra_g+\eta\Delta_\phi s\\
		&=(\eta'\Delta_\phi r+\eta'')s+2\eta'\la \n_gr,\n_gs\ra_g+\eta \Delta_\phi s.
		\end{split}
		\end{equation}
		By the Laplace comparison theorem for the Bakry-\'Emery tensor (Theorem~\ref{th:mean_curvature_I}), in the region $r>\delta$ we have
		\begin{equation}\label{eq:laplace_comparison}
		\Delta_\phi r=m_\phi(x)\leq \sup \{m_\phi(y)\ |\ {y\in \pd B_{\delta}(p_0)}\}=m_\delta
		\end{equation}
		Therefore, since $-c_0(\delta)< \eta'\leq 0$, and $\eta'(t)=0$ for $t\in(0,\delta)$, we necessarily have in the entire ball $B_R(p_0)$ that
		\[
		\eta'\Delta_\phi r\geq -c_0(\delta)m_\delta
		\]
		Furthermore, by~\eqref{eq:pf_1st_order} at $p\in B_R(p_0)\subset M$ we have $\eta'\la \n_gr,\n_gs\ra_g=-\eta^{-1}(\eta')^2 s$, hence with the use of the differential inequality $\Delta_\phi s\geq s^2$, we find
		\[
		\begin{split}
		0\geq \eta^{-1}(-m_\delta c_0(\delta)+\eta'')\Phi(p)-2\eta^{-2}(\eta')^2\Phi(p)
		+\eta^{-1}\Phi(p)^2
		\end{split}
		\]
		It follows that for our choice of the cut-off function
		\[
		\Phi(p)\leq m_\delta c_0(\delta)-\eta''+2\eta^{-1}(\eta')^2\leq C(\delta,m_\delta )
		\]
		Therefore, since $\Phi$ attains its maximum at $p$, at the base point $p_0$ we have
		\[
		\Phi(p_0)=R {s(p_0)}\leq \Phi(p)\leq C(\delta,m_\delta)
		\]
		and $s(p_0)\leq C(\delta,m_\delta)R^{-1}$ as claimed.

		If $(N,g)$ is complete, then for any $p_0\in N$ and $R>0$ there is the geodesic ball $B_R(p_0)\subset N$ so that $s(p_0)\leq C(\delta,m_\delta)R^{-1}$. Therefore we conclude $s(p_0)=0$ and the function $s(x)$ vanishes identically.
	\end{proof}

	\section{Geometry of the Monge-Amp\`ere equation}\label{s:ma}

	Let $(N,D)$ be a \textit{simply-connected} Hessian manifold with the globally defined Hessian metric
	\[
	g=D^2u,\quad u\in C^\infty(N,\R).
	\]
	Fix a development map $\bm{\iota}\colon N\to \R^n$ into an $n$-dimensional space with the coordinates $\{x^1,\dots,x^n\}$. Differential forms $\{\bm\iota^*dx^i\}$ provide a $D$-parallel trivialization of $T^*N$. Abusing the notation slightly we will drop symbol $\bm{\iota}^*$ and treat $\{dx^i\}$ as differential forms on $N$, and $\{x^i\}$ as a local coordinate system. For a function(tensor) $T$, we denote by subscript $T_{,i}$ the derivative $D_{\pd/\pd x^i}T$. We also write $g_{ij}:=u_{,ij}$ for the components of the metric tensor $g$ relative to the basis $\{dx^i\}$ and $g^{ij}$ for the components of its inverse $g^{-1}$, so, e.g., $g^{ij}u_{,jk}=\delta_{k}^i$ and $g_{ij,k}=g_{ik,j}=u_{,ijk}$.

	Assume that the smooth convex function $u\colon N\to \R$ defining the metric $g$ satisfies equation
	\begin{equation}
	\log\det D^2u=-\la v, x\ra + \la du,\xi\ra + c,\qquad \xi\in \R^n, v\in (\R^n)^*.
	\end{equation}
	Equivalently, in the coordinates $\{x^i\}$, function $u(x)$ solves the Monge-Amp\`ere equation
	\begin{equation}\label{eq:main}
	\log\det (u_{,ij}) = -v_jx^j + u_{,i}\xi^i + c.
	\end{equation}

	We have the following immediate differential identity as a consequence of~\eqref{eq:main}.

	\begin{lemma}[Key differential identity]\label{lm:diff_id}
		Assume that function $u$ solves the Monge-Amp\`ere equation~\eqref{eq:main}. Then $u$ satisfies the following identity:
		\begin{equation}\label{eq:1order}
		g^{qp}u_{,pqi}=-v_{i}+u_{,iq}\xi^q
		\end{equation}
	\end{lemma}
	\begin{proof}
		Recall that if $A=A(t)$ is a smooth matrix-valued function, such that $A(t)$ is invertible, then $(\det A)'=\det A \cdot \tr (A^{-1}A')$. Therefore, applying operator $D_{\pd/\pd x^\alpha}$ to both sides of~\eqref{eq:main} we find
		\[
		g^{qp}u_{,pqi}=-v_{i}+u_{,iq}\xi^q
		\]
		as claimed.
	\end{proof}

	Now we express basic geometric quantities with respect to $g$ in terms of the local affine coordinates $\{x^i\}$ and the function $u$.

	\begin{lemma}[Riemannian quantities]\label{lm:riem_id}
		Let $(N,D)$ be an affine manifold with a Hessian metric $g=D^2u$ for $u(x)\in C^\infty(N,\R)$. The Christoffel symbols of $g_{ij}$ are given by
		\begin{equation}\label{eq:christoffel}
		\Gamma_{ij}^k=\frac{1}{2}g^{lk}u_{,ijl}
		\end{equation}
		The full Riemannian tensor of $g_{ij}$ is
		\begin{equation}\label{eq:riem}
		\Rm_{ijkl}=\frac{1}{4}g^{pq}(u_{,ilp}u_{,jkq}-u_{,jlp}u_{,ikq}).
		\end{equation}
		The Ricci tensor of $g_{ij}$ is
		\begin{equation}\label{eq:ricci}
		\Ric_{ij}=g^{jl}\Rm_{ijkl}=\frac{1}{4}g^{pq}g^{mn}\left(u_{,ipm}u_{,jqm}-u_{,ijq}u_{,pmn}\right)
		\end{equation}
		The scalar curvature of $g_{ij}$ is
		\begin{equation}\label{eq:scalar}
		s=\frac{1}{4}(|u_{,ijk}|_g^2-|g^{mn}u_{,imn}|_g^2)
		\end{equation}
		If furthermore, $u(x)$ solves the Monge-Amp\`ere equation~\eqref{eq:main}, then we have additionally
		\begin{equation}
		\begin{split}
			\Gamma_{ik}^k&=(\log |g|^{1/2})_{,i}=\frac{1}{2}(u_{,ij}\xi^j-v_i)\\
			\Ric_{ij}&=\frac{1}{4}g^{pq}g^{mn}u_{,ipm}u_{,jqm}-\frac{1}{4}(u_{,ijp}\xi^p-g^{pq}u_{,ijq}v_p)\\
			s&=\frac{1}{4}\left(|u_{,ijk}|_g^2-|u_{,ij}\xi^j-v_i|^2_g\right)
		\end{split}
		\end{equation}
	\end{lemma}
	\begin{proof}
		Keeping in mind that $g_{ij}=u_{,ij}$, the identity $\Gamma_{ij}^k=\frac{1}{2}g^{kl}\left(g_{jl,k}+g_{il,j}-g_{ij,l}\right)$ immediately implies the formula for the Christoffel symbols. With the expression for the Christoffel symbols at hand, the formulas for the Riemannian, Ricci and scalar curvatures are a matter of straightforward computations. These formulas hold for any Hessian metric.

		Now, if $u(x)$ solves the Monge-Amp\`ere equation, then by the differential identity of Lemma~\ref{lm:diff_id}
		\[
		g^{qp}u_{,pqi}=-v_i+u_{,iq}\xi^q
		\]
		so we have
		\[
		\Gamma_{ik}^k=\left(\log \det g^{1/2}\right)_{,i}=\frac{1}{2}(-v_i+u_{,iq}\xi^q).
		\]
		This immediately implies the refined formulas for $\Ric$ and $s$.
	\end{proof}

	\begin{remark}
		In the case of the standard Monge-Amp\`ere equation $\log\det u_{,ij}=c$, we have $v=\xi=0$, and the above calculations imply that $\Ric\geq 0$. This was the essential ingredient in the foundational work of Calabi~\cite{ca-58} allowing for what is now known as \textit{Calabi's third order estimate}. Unlike the ``standard'' setting $v=\xi =0$, we do not have a favorable sign for the Ricci curvature corresponding to the solution $u(x)$ of the Monge-Amp\`ere equation~\eqref{eq:main}. This makes the main steps of~\cite{ca-58} invalid. The crucial fact is that we still have non-negativity of the \textit{Bakry-\'Emery} Ricci tensor $\Ric_\phi$ for an appropriate choice of function $\phi$. This observation was made by Kolesnikov~\cite[Theorem\,1.1]{ko-14} in a more general setting, when the linear functions $V(x)=v_ix^i$ and $W(y)=u_{,i}\xi^i=y_i\xi^i$ are replaced by arbitrary convex functions. For the readers convenience, we give a direct proof in our special setting.
	\end{remark}

	\begin{lemma}[Hessian \& Laplace operator]
		Assume that function $u(x)$ solves~\eqref{eq:main}. Let $\n_g$ and $\Delta$ be the Levi-Civita connection and the Laplace-Beltrami operator associated with $g=D^2u$. Then for $\phi\in C^\infty(N,\R)$
		\begin{equation}\label{eq:hessian}
		(\n_g^2\phi)_{ij}=\phi_{,ij}-\Gamma_{ij}^k \phi_{,k}=\phi_{,ij}-\frac{1}{2}g^{lk}u_{,ijl}\phi_{,k}
		\end{equation}

		\begin{equation}\label{eq:laplace}
		\begin{split}
		\Delta_g \phi&=
		g^{pq}\phi_{,pq}-\frac{1}{2}\xi^p\phi_{,p}+\frac{1}{2}g^{pq}v_p\phi_{,q}
		\end{split}
		\end{equation}
	\end{lemma}
	\begin{proof}
		The Hessian formula follows immediately from the expression for the Christoffel symbols of Lemma~\ref{lm:riem_id}. The Laplacian formula~\eqref{eq:laplace} is then obtained by taking the trace of~\eqref{eq:hessian} and using the first differential identity of Lemma~\ref{lm:diff_id}.
	\end{proof}

	Now we are ready to prove the key geometric property of the solutions to the Monge-Amp\`ere equation~\eqref{eq:main}.

	\begin{proposition}[Positivity of the Bakry-\'Emery Ricci {tensor,~\cite[Theorem\,1.1]{ko-14}}]\label{prop:ric_positive}
		Let $(N,D)$ be a simply-connected affine manifold equipped with a development map $\bm{\iota}\colon N\to\R^n$ and a Hessian metric $g=D^2u(x)$, such that function $u(x)$ solves the Monge-Amp\`ere equation
		\[
		\log\det D^2u=-\la v,x\ra+\la du,\xi\ra+c,\quad \xi\in\R^n, v\in (\R^n)^*.
		\]
		Consider the function
		\[
		\phi=\frac{1}{2}\left(u_{,p}\xi^p+v_qx^q\right).
		\]
		Then
		\[
		(\Ric_\phi)_{ij}=\frac{1}{4}g^{pq}g^{kl}u_{,ipk}u_{,jql}.
		\]
		In particular, $\Ric_\phi\geq 0$.
	\end{proposition}
	\begin{proof}
		Computing $\n^2_g\phi$ for $\phi=\frac{1}{2}\left(u_{,i}\xi^i+v_ix^i\right)$ we have
		\[
		\begin{split}
		(\n^2_g\phi)_{ij}&=\phi_{,ij}-\frac{1}{2}g^{lk}u_{,ijl}\phi_{,k}\\&=\frac{1}{2}u_{,pij}\xi^p-\frac{1}{4}(g^{lk}u_{,ijl}u_{,pk}\xi^p+g^{lk}u_{,ijl}v_k)
		\\&=\frac{1}{4}u_{,pij}\xi^p-\frac{1}{4}g^{lk}u_{,ijl}v_k,
		\end{split}
		\]
		where we used that $g^{lk}u_{,pk}=\delta_p^l$. With this calculation and the identities of Lemma~\ref{lm:diff_id} we find
		\[
		(\Ric_\phi)_{ij}=(\Ric+\n^2_g\phi)_{ij}=\frac{1}{4}g^{pq}g^{mn}u_{,ipm}u_{,jqn}\geq 0
		\]
		as claimed.
	\end{proof}

	Proposition~\ref{prop:ric_positive} suggests that the properties of the solutions to the Monge-Amp\`ere equation~\eqref{eq:main} are closely related to the geometry of the smooth metric measure space $(N,g,e^{-\phi}dV_g)$. Supporting this intuition, we give a modified proof of Calabi's third order estimate taking into account the presence of a non-trivial weight function $\phi$. We will closely follow the calculations of Calabi~\cite{ca-58} in the case $v=\xi=0$ making the necessary modifications allowing for nonzero $v$ and $\xi$ in~\eqref{eq:main}. A more straightforward coordinate proof of this inequality can be found in~\cite[\S7]{ko-14}.


	\begin{proposition}\label{prop:diff_ineq_ma}
		Let $(N,D)$ be a Hessian manifold with $g=D^2u$ as above. Assume that $u(x)$ solve the Monge-Amp\`ere equation~\eqref{eq:main} and let
		\[
		\phi=\frac{1}{2}(u_{,p}\xi^p+v_qx^q).
		\]
		Then function
		\[
		\sigma:=|u_{,ijk}|^2
		\]
		satisfies the differential inequality
		\begin{equation}\label{eq:w_laplace_ma}
		\Delta_\phi \sigma\geq \frac{1}{2n}\sigma^2.
		\end{equation}
	\end{proposition}
	\begin{proof}
		Up to a factor of 4, the quantity $\sigma$ can be thought of either as the ``weighted'' scalar curvature $\sigma=4\tr_g( \Ric_\phi)=s+\Delta \phi$, or, equivalently, as the quantity measuring the difference between the two connections $D$ and $\n_g$ canonically attached to $(N,D,g)$:
		\[
		\sigma=4|D-\n_g|^2.
		\]

        For convenience, we choose an orthonormal frame at a given point $p\in N$, so that $g_{ij}=\delta_{ij}$, and in the calculations below denote $\n:=\n_g$. Treating $u_{,ijk}$ as a section of $\Sym^3(T^*N)$, we immediately have
        \begin{equation}\label{eq:pf_ineq0}
        \Delta|u_{,ijk}|^2=2|\n_lu_{,ijk}|^2+2\la \n_l\n_l u_{,ijk},u_{,ijk}\ra_g.
        \end{equation}
        The first term is nonnegative. To compute the second term we observe that the tensor $        \n_lu_{,ijk}=u_{,ijkl}-\frac{1}{2}\left(
        u_{,ijs}u_{,skl}+u_{,isk}u_{,sjl}+u_{,sjk}u_{,sil}.
        \right)
        $ is symmetric in its 4 arguments, in particular
        \begin{equation}\label{eq:A_sym}
        \n_l u_{,ijk}=\n_i u_{,ljk}
        \end{equation}
        Therefore, we can rewrite the rough Laplacian term as follows:
        \begin{equation}\label{eq:pf_ineq1}
        \begin{split}
            \n_l\n_lu_{,ijk}
            &=\n_l\n_i u_{,ljk}=\n_i\n_lu_{,ljk}+(\n_l\n_i-\n_i\n_l)u_{,ljk}\\
            &=\n_i(\n_ju_{,kll})+\Ric_{is}u_{,sjk}+\Rm_{lisj}u_{,lsk}+\Rm_{lisk}u_{,ljs}
        \end{split}
        \end{equation}
        Using the differential identity of Lemma~\eqref{lm:diff_id}, and recalling that $\phi=\frac{1}{2}(v_px^p+u_{,q}\xi^q)$, we compute the first term as
        \begin{equation}\label{eq:pf_ineq2}
        \begin{split}
        \n_i(\n_ju_{,kll})
        &=\n_i\n_j(-v_k+\xi^k)=\frac{1}{2}\n_i(u_{,jks}v_s+u_{,jks}\xi^s)\\
        &=\n_iu_{,jks}\frac{1}{2}(v_s+\xi^s)+u_{,jks}\frac{1}{2}\n_i(v_s+\xi_s)\\
        &= \n_su_{,ijk} (d\phi)_s+u_{,jks}(\n^2\phi)_{is}.
        \end{split}
        \end{equation}
        Substituting~\eqref{eq:pf_ineq2} into~\eqref{eq:pf_ineq1}, and recalling that $\Ric_\phi=\Ric+\n^2\phi$, we obtain
        \[
        \n_l\n_lu_{,ijk}=\n_su_{,ijk}(d\phi)_s+(\Ric_\phi)_{is}u_{,jks}+\Rm_{lisj}u_{,lsk}-\Rm_{liks}u_{,ljs}.
        \]
        Now, evaluating the inner product with $u_{,ijk}$ we find
        \[
        \la\n_l\n_l u_{,ijk},u_{,ijk}\ra=\frac{1}{2}\la d|u_{,ijk}|^2,d\phi\ra+4|\Ric_\phi|^2+4|\Rm|^2,
        \]
        where we used that $(\Ric_\phi)_{ij}=\frac{1}{4}u_{,ipq}u_{,jpq}$ and $\Rm_{ijkl}=\frac{1}{4}(u_{,ilp}u_{,jkp}-u_{,jlp}u_{,ikp})$. Plugging it back into~\eqref{eq:pf_ineq0} we get
        \[
        \Delta|u_{,ijk}|^2=2|\n_lu_{,ijk}|^2+\la d|u_{,ijk}|^2,d\phi\ra+8|\Ric_\phi|^2+8|\Rm|^2.
        \]
        Therefore
        \[
        \Delta_\phi|u_{,ijk}|^2\geq 8|\Ric_\phi|^2.
        \]
        It remains to note that
        \[
        |\Ric_\phi|^2\geq \frac{1}{n}\tr_g(\Ric_\phi)^2=\frac{1}{16 n}(|u_{,ijk}|^2)^2
        \]
        so that $\Delta_\phi|u_{,ijk}|^2\geq \frac{1}{2n}(|u_{,ijk}|^2)^2 $ as required.
	\end{proof}
    \begin{remark}
        The conclusion of Proposition~\ref{prop:diff_ineq_ma} is somewhat weaker than the analogous estimate of Calabi~\cite[Prop.\,1]{ca-58}. There are two reasons for it: first we completely discarded a nonnegative term $|\n_lu_{,ijk}|^2$, which could have been used to strengthen Proposition~\ref{prop:diff_ineq_ma}. Second, more importantly, unlike the classical Riemannian setting, tensors $\Ric_\phi$ and $\Rm$ are not directly related and we cannot bound $|\Ric_\phi|^2$ by a multiple of $|\Rm|^2$ without some a priori information about $\phi$.
    \end{remark}

    \begin{theorem}\label{th:ma_rigid}
        Let $(N,D)$ be a simply-connected affine manifold equipped with a development map $\bm{\iota}\colon N\to\R^n$ and a Hessian metric $g=D^2u(x)$, such that function $u(x)$ solves the Monge-Amp\`ere equation
        \[
        \log\det (u_{,ij}) = -v_jx^j + u_{,i}\xi^i + c,\quad \xi\in\R^n, v\in (\R^n)^*.
        \]
        If the manifold $(N,g)$ is complete, then $\bm\iota\colon N\to\R^n$ is a diffeomorphism, and $u(x)$ is a quadratic polynomial. In particular $(N,D,g)$ is equivalent as a Hessian manifold to $\R^n$ equipped with a flat metric $g$.
    \end{theorem}
    \begin{proof}
        By Proposition~\ref{prop:diff_ineq_ma}, function $\sigma=|u_{,ijk}|^2$ satisfies the differential inequality $\Delta_\phi\sigma\geq\frac{1}{2n}\sigma^2$, or, equivalently,
        \[
        \Delta_\phi(\sigma/2n)\geq (\sigma/2n)^2.
        \]
        Since $(N,g)$ is complete, by the Liouville-type Theorem~\ref{th:lioville} we conclude that $\sigma=|u_{,ijk}|^2$ is identically zero, so that $u(x)$ is a quadratic polynomial in $\{x^i\}$, and $g=D^2u(x)$ is a pullback of a flat metric $g_0$ on $\R^n$.

        We claim that $\bm\iota(N)=\R^n$. Indeed, let $p_0\in \bm\iota(N)\subset \R^n$ and pick any other point $p\in \R^n$. Connect $p_0$ and $p$ with a straight segment $[p_0,p]$ in $\R^n$. Since $[p_0,p]$ has a finite $g_0$-length, $N$ is complete and $\bm\iota$ is a local isometry we conclude that $[p_0,p]$ must be entirely contained in $\bm\iota(N)$.
        Finally, since $N$ is connected, $\R^n$ is simply-connected, and $\bm\iota\colon N\to\R^n$ is a local diffeomorphism, $\bm\iota$ must be a bijection.
    \end{proof}

    In a special case when $N\subset \R^n$ is an open subset we have the following corollary generalizing the result of Calabi~\cite{ca-58} about solutions to $\det u_{,ij}=1$:

    \begin{corollary}\label{cor:calabi_gen}
        Let $N\subset \R^n$ be an open subset and $u\colon N\to  \R$ a strictly convex function solving the Monge-Amp\`ere equation
        \begin{equation}
            \log\det (u_{,ij})=-v_jx^j+u_{,i}\xi^i+c,
        \end{equation}
        where $v\in (\R^n)^*$, $\xi\in\R^n$ and $c$ is a constant. If metric $g=D^2u$ on $N$ is complete, then $N=\R^n$, and $u(x)$ is a quadratic polynomial.
    \end{corollary}

    \section{Uniqueness of toric K\"ahler-Ricci solitons}\label{s:main}

    We are now ready to prove our main result. Recall, that a gradient steady toric K\"ahler-Ricci soliton is K\"ahler manifold $(M,g,J )$ equipped with an effective Hamiltonian action of a half-dimensional torus $\T^n$ solving the steady K\"ahler-Ricci soliton equation
    \begin{equation}\label{eq:s5_krs}
    \rho_\omega+\i\pd\bar\pd f=0.
    \end{equation}
    By averaging the soliton potential we can assume that $f$ is $\T^n$-invariant so that the Killing holomorphic vector field $J\n_g f$ is the infinitesimal vector field $X_v$ for some $v\in\mf t$. On the locus $M_0\subset M$ where the action $\T^n\colon M\to M$ is free, equation~\eqref{eq:s5_krs} is locally equivalent to the Monge-Amp\`ere for a function $u(x)$ on $U\subset \mf t^*$:
    \[
    \log\det u_{,ij}=-v_px^p+u_{,q}\xi^q+c.
    \]

    We further assume that $(M,g)$ is complete. In general, when the action $\T^n\times M\to M$ has nontrivial stabilizers, the orbit space $N:=M/\T^n$ is not a manifold,
    and the analytical tools developed in Section~\ref{s:ma} only apply to the open dense subset $N_0=M_0/\T^n$. In the presence of non-trivial stabilizers, $N_0$ will always be incomplete, so the rigidity Theorem~\ref{th:ma_rigid} does not hold. However, in the particular setting when the action $\T^n\times M\to M$ is \textit{free}, we have $N_0=N$ is complete which allows us to deduce the following result.

    \begin{theorem}\label{th:krs_uniqueness}
       Let $(M,g,J ,\T^n)$ be a complete steady toric K\"ahler-Ricci soliton with the free torus action $\T^n\times M\to M$. Then $(M,g,J ,\T^n)$ is equivariantly isomorphic to a flat algebraic torus $(\C^*)^n$ equipped with the standard complex structure and the action of $U(1)^n\subset (\C^*)^n$.
    \end{theorem}

    \begin{proof}
        Let $\bm\pi\colon M\to N$ be the projection onto the orbit space. The orbit space is a smooth manifold, since the action of $\T^n$ is free. Next $N$ inherits an affine structure $(N,D)$ from the moment map $\til{\bm\mu}\colon N\to \mf t^*$ which is the development map for $D$. By Corollary~\ref{cor:hessian} $N$ also has a natural Hessian structure $(N,D,\bm\pi_*g)$. Furthermore since $(M,g)$ is complete, the orbit space $(N,\bm\pi_*g)$ is complete as well. Let $\til N\to N$ be the universal cover of $N$, and by abuse of notation denote by
        \[
        \til{\bm\mu}\colon \til N\to \mf t^*
        \]
        the lift of the moment map to $\til N$. Corollary~\ref{cor:MA_eq_simply_connected} guarantees that there exists a function $u\colon \til N\to \R$ and $v\in\mf t$, $\xi\in\mf t^*$, $c\in \R$ such that
        \[
        \log\det u_{,ij}=-v_px^p+u_{,q}\xi^q+c.
        \]
        Now we can apply Theorem~\ref{th:ma_rigid} to $(\til N,D,\bm\pi_*g)$ and conclude that $\til{\bm\mu}\colon \til N\to\mf t^*$ is a diffeomorphism, and $u(x)$ is a quadratic polynomial. As $\til{\bm\mu}\colon \til N\to \mf t^*$ is constant on the orbits of the deck transformation of the covering $\til N\to N$, we conclude that $\til N=N$. Therefore the functions $\{x^i\}$ as well as $\{y_i=u_{,i}\}$ provide a globally defined coordinate system on $N$. Since $u$ is a positive definite quadratic polynomial, the range of $(y_1,\dots,y_n)$ is $\R^n$. Now, as $N$ is simply-connected, we have the globally defined angular coordinates $\theta_i\in U(1)$ on $M$ such that $J^*d\theta_i=dy_i$. Thus $M=\mf t^*\times \T^n\simeq (\C^*)^n$ with the global holomorphic coordinates $\{y_i+\i\theta_i\}$. Finally, metric $g$ on $M$ expressed in the coordinates $\{x^i\}$, $\{\theta_i\}$
        \[
        g=\sum_{i,j=1}^n \bm G_{ij}dx^idx^j+\sum_{i,j=1}^n \bm (\bm G^{-1})^{ij}d\theta_id\theta_j,\quad \bm G_{ij}=u_{,ij}\]
        has constant $\bm G$ and thus is flat.
    \end{proof}

	Theorem~\ref{th:krs_uniqueness} is essentially a non-existence result implying that there are no non-trivial complete toric KRS on $(\C^*)^n$. It would be interesting to relax the freeness assumption, and obtain consequences about the geometry of general toric KRS. Clearly, we cannot expect to have a similar uniqueness/classification result, since already $\C^n=\prod \C^{n_i}$ admits many non-equivalent toric KRS given by the products of Cao's solitons $\C^{n_i}$. A technical condition particularly important in the theory of non-compact toric K\"ahler manifolds (see~\cite{pr-wu-94}) is the \textit{properness} of the moment map
	\[
	\bm{\mu}\colon M\to \mf t^*,
	\]
	i.e., $\bm\mu(K^{-1})$ is compact for a compact $K\subset \mf t^*$. Crucially, it allows to set up the Duistermaat-Heckman theory on $M$, which in turn is instrumental in defining the weighted analogous of the Aubin and other energy functionals, see, e.g.,~\cite{ci-20,co-de-su-19}.
	
	We conjecture that the moment map on a complete toric KRS is necessarily proper, and expect that the results of Sections~\ref{s:bakry-emery} and~\ref{s:ma} about the Monge-Amp\`ere equations will be useful in proving it.
	\begin{conjecture}
		Let $(M,g,J ,\T^n)$ be a complete toric K\"ahler-Ricci soliton. Then the moment map $\bm\mu\colon M\to \mf t^*$ is proper.
	\end{conjecture}
	Properness of the soliton potential $f=\la \bm \mu,v\ra+c$ implies the properness of $\bm{\mu}$, and the former is known to hold on the shrinking Ricci solitons, and under mild geometric assumptions on the expanding solitons, see~\cite{co-de-su-19}. Unfortunately, on a \textit{steady} KRS the soliton potential $f$ is not necessarily proper, providing no information about the moment map.

		%
	\emergencystretch=1em
	\printbibliography
\end{document}